\documentclass[11pt]{article}

\usepackage{amsfonts}
\usepackage{amsmath}

\newcommand{\beq}{\begin{equation}}
\newcommand{\eeq}{\end{equation}}
\newcommand{\bea}{\begin{eqnarray}}
\newcommand{\eea}{\end{eqnarray}}
\newcommand{\beas}{\begin{eqnarray*}}
\newcommand{\eeas}{\end{eqnarray*}}

\newtheorem{theorem}{Theorem}[section]

\newtheorem{conjecture}[theorem]{Conjecture}
\newtheorem{definition}[theorem]{Definition}
\newtheorem{proposition}[theorem]{Proposition}

\newtheorem{corollary}[theorem]{Corollary}
\newtheorem{lemma}[theorem]{Lemma}
\newtheorem{remark}[theorem]{Remark}
\newtheorem{example}[theorem]{Example}
\newtheorem{examples}[theorem]{Examples}
\newtheorem{foo}[theorem]{Remarks}
\newtheorem{hypothesis}[theorem]{Hypothesis}

\newenvironment{proof}{\addvspace{\medskipamount}\par\noindent{\it
Proof}.}
{\unskip\nobreak\hfill$\Box$\par\addvspace{\medskipamount}}

\newcommand{\p}{\partial}
\newcommand{\ee}{\ell}

\newcommand{\bM}{\mathbb M}

\newcommand{\M}{\mathbb M}

\newcommand{\R}{\mathbb R}

\newcommand{\Q}{\mathbf{Tr}_{\mathcal{H}}Q}

\title{Curvature dimension inequalities and subelliptic heat kernel gradient bounds on contact  manifolds}
\author{Fabrice Baudoin\footnote{First author supported in part by NSF Grant DMS 0907326}, Jing Wang}

\date{Department of Mathematics, Purdue University \\
 West Lafayette, IN, USA}

\begin{document}
\maketitle

\begin{abstract}
We study curvature dimension inequalities for the sub-Laplacian on contact Riemannian manifolds. This new curvature dimension  condition is then used to obtain:
\begin{itemize}
\item Geometric conditions ensuring the compactness of the underlying manifold (Bonnet-Myers type results);
\item Volume estimates of  metric balls;
\item Gradient bounds and stochastic completeness for the heat semigroup generated by the sub-Laplacian;
\item Spectral gap estimates.
\end{itemize}
\end{abstract}

\tableofcontents

\section{Introduction}

 Let $(\M, \theta,g)$ be a $2n+1$ smooth contact Riemannian manifold. On $\M$, there is a canonical diffusion operator $L$: The contact sub-Laplacian. This operator is not elliptic but only subelliptic in the sense of Fefferman-Phong \cite{FP} (see also \cite{JSC} for a survey on subelliptic diffusion operators). 
 
 This lack of ellipticity makes the study of the geometrically relevant functional inequalities associated   to $L$ particularly delicate. Some methods have been developed in the literature but are local in nature (see \cite{CY}, \cite{Jer}, \cite{SC})  and no global methods were known before the work by Baudoin-Garofalo \cite{BG1}, except in the three dimensional case (see \cite{A}, \cite{R}). One of the main obstacles is the complexity of the theory of Jacobi vector fields (see \cite{LZ}).
 
 \
 
 In the work \cite{BG1},  instead of dealing with Jacobi  fields, the authors use the Bochner's method and proved that, if $\M$ is Sasakian, then under some geometric conditions (a lower bound on the Ricci curvature tensor of the Tanaka-Webster connection), the operator $L$ satisfies a generalized curvature dimension inequality that we now describe. On $\M$, there is a canonical vector field, the Reeb vector field $Z$ of the contact form $\theta$, it is transverse to the kernel of $\theta$.

Given the sub-Laplacian $L$ and the first-order bilinear forms 
\[
\Gamma(f)=\frac{1}{2} \left( L(f^2)-2fLf \right),
\]
 and 
 \[
 \Gamma^Z(f)=(Zf)^2,
 \]
we can introduce the following second-order differential forms:
\begin{equation}\label{gamma2}
\Gamma_{2}(f,g) = \frac{1}{2}\big[L\Gamma(f,g) - \Gamma(f,
Lg)-\Gamma (g,Lf)\big],
\end{equation}
\begin{equation}\label{gamma2Z}
\Gamma^Z_{2}(f,g) = \frac{1}{2}\big[L\Gamma^Z (f,g) - \Gamma^Z(f, Lg)-\Gamma^Z (g,Lf)\big].
\end{equation}

The following basic result connecting the geometry of the contact manifold $\M$ to the analysis of its sub-Laplacian was then proved in \cite{BG1}. It requires the contact structure on $\M$ to be of Sasakian type: A class of contact manifolds that contain very interesting examples (see \cite{BW}, \cite{JW}) but that is somehow restrictive.

\begin{theorem}\label{T:sasakiani}
Let $(\bM,\theta)$ be a complete Sasakian contact manifold  with  dimension $2n+1$. The Tanaka-Webster Ricci tensor satisfies the bound  
\[
\mathbf{Ric}_x(v,v)\ \ge \rho_1|v|^2, x \in \M , v\in \mathbf{Ker} (\theta),
\]
if and only if for every smooth and compactly supported function $f$,
 \begin{align}\label{CDnonlineaire}
 \Gamma_2(f) +2\sqrt{ \Gamma(f) \Gamma_2^Z(f) }\ge \frac{1}{2n}(Lf)^2 +\rho_1 \Gamma(f) +\frac{n}{2} \Gamma^Z(f).
 \end{align}
\end{theorem}

Observe that by linearization, the inequality (\ref{CDnonlineaire}) is equivalent to the fact that for every $\nu >0$,
\begin{align}\label{CDSasaki}
\Gamma_2(f)+\nu\Gamma_2^Z(f) \ge \frac{1}{2n}(Lf)^2+\left(\rho_1-\frac{1}{\nu}  \right)\Gamma(f) +\frac{n}{2} \Gamma^Z(f).
\end{align}
Theorem \ref{T:sasakiani} opened the door to the study of global functional inequalities on Sasakian manifolds, like the  log-Sobolev inequalities (see \cite{BBlogsob}), the Sobolev and isoperimetric inequalities (see \cite{BK}), the Li-Yau type gradient bounds for the heat kernel (see \cite{BG1}) and the Gaussian upper and lower bounds for the heat kernel (see \cite{BBG}). These inequalities were obtained through a systematic use of the heat semigroup associated to $L$ and Bakry-\'Emery type computations \cite{bakry-stflour}, \cite{Bakry-Emery}.

\

Our goal in this paper is to remove the assumption that $\M$ is a Sasakian manifold. The Sasakian condition is equivalent to the fact that the contact manifold carries a CR structure and that the Reeb vector field acts isometrically on the kernel of $\theta$. This condition is equivalent to the fact that the forms $\Gamma$ and $\Gamma^Z$ are intertwined in the sense that
\begin{align}\label{sasaki}
\Gamma(f,\Gamma^Z(f))=\Gamma^Z(f,\Gamma(f)).
\end{align}
The condition is restrictive and many interesting examples of contact manifolds are not Sasakian. It is thus interesting to see if the Sasakian condition can be dropped. Our main result in that direction is the following theorem that shows the structure of the curvature dimension condition in the most general class of contact manifolds:

\begin{theorem}(See Theorem \ref{CD-geo-thm})
Let $(\M,\theta,g)$ be a $2n+1$-dimensional contact Riemannian manifold. If some geometric conditions are satisfied, then there exist constants $\rho_1, \rho_2$ and $\rho_3$ such that for every $\nu >0$ and smooth and compactly supported function $f$:
\begin{equation}\label{eqCDndintro}
\Gamma_2(f)+\nu\Gamma^Z_2(f)\geq
\frac{1}{2n}(Lf)^2+\left(\rho_1-\frac{1}{\nu}\right)\Gamma(f)+(\rho_2-\rho_3\nu^2)\Gamma^Z(f).
\end{equation}
\end{theorem}

 The main difference with the Sasakian curvature dimension condition \eqref{CDSasaki} is therefore the appearance of the strongly nonlinear term $-\rho_3 \nu^2 \Gamma^Z(f)$. It is noticeable that this new curvature dimension inequality appears as a special case of a general class of inequalities that was recently proposed in an abstract setting by F.Y. Wang in \cite{FYW}. Our approach here is more of geometric nature, in the sense that our goal is to precisely understand what are the geometric bounds that imply a curvature dimension condition. As a consequence we get a very explicit curvature dimension condition.
 
 As we will see, the new term makes the curvature dimension condition much more difficult to exploit. However,  we can still address the following questions by using our new curvature dimension inequality:
 \begin{enumerate}
 \item \textbf{Bonnet-Myers type results} (See Theorem \ref{myers}). We provide geometric conditions ensuring the compactness of $\M$;
 \item \textbf{Volume estimate} (See Proposition \ref{stochastic complete}) We prove that under suitable geometric conditions the volume of balls has at most an exponential growth; 
\item \textbf{Stochastic completeness of the heat semigroup associated to the contact sub-Laplacian} (See Proposition \ref{stochastic complete}).  We prove that if the curvature dimension inequality (\ref{eqCDndintro})  and an additional  condition are satisfied, then the semigroup is stochastically complete.  
\item \textbf{Poincar\'e inequality} (See Theorem \ref{Poincare}). By using the generalized curvature dimension inequality to prove gradient bounds for the heat semigroup, we show that if  \eqref{eqCDndintro} is satisfied  with $\rho_1-\frac{\kappa\sqrt{\rho_3}}{\sqrt{\rho_2}} > 0$, then for every smooth and compactly supported function $f$ on $\M$:
\[
 \int_\M f^2d\mu-\left(\int_\M fd\mu \right)^2\leq\frac{\rho_2+\kappa}{\rho_1\rho_2-\kappa\sqrt{\rho_2\rho_3}}\int_\M\Gamma(f)d\mu. 
 \]
As a consequence, $-L$ has a spectral gap of size bigger than  $ \frac{\rho_1\rho_2-\kappa\sqrt{\rho_2\rho_3}}{{\rho_2+\kappa}}$.
\end{enumerate}

The paper is organized as follows. In Section \ref{contact}, we introduce the geometric prerequisites that are needed in this work. Section \ref{Bochners} is devoted to a careful analysis of the Bochner's type formulas that are needed to establish the generalized curvature dimension condition (\ref{eqCDndintro}). Bochner's type formulas on CR manfolds have been extensively studied in the literature (see for instance \cite{B}, \cite{CY}, \cite{Chiu}, \cite{G}, \cite{I}, \cite{LL}). The horizontal Bochner's formula we obtain in Theorem \ref{T:bochner}  is an extension of the CR Bochner fomula of the above mentioned works since we work in the more general framework of an abritary Riemannian contact manifold for which Tanno's tensor is not necessary zero. As it is well-known in the CR case, this Bochner's formula makes appear a second order differential term involving a differentiation in the vertical direction of the Reeb vector field. This term is the main source of difficulties, since it may not be bounded in terms of the horizontal gradient. The main idea is then to prove a vertical Bochner's formula: This is our Theorem  \ref{T:vertical}. Computations show then that the \textit{annoying} second order differential term of the horizontal Bochner's formula also appears in the vertical Bochner's formula. As a consequence,  the horizontal and the vertical Bochner's formulas perfectly match together and a linear combination of them produces the curvature-dimension inequality \eqref{eqCDndintro}.

\

In Section \ref{Bonnets}, we apply the generalized curvature dimension inequality to the study of the stochastic completeness of the subelliptic heat semigroup and to the problem of the compactness of the manifold under suitable geometric conditions. The main idea is that the generalized curvature dimension inequality \eqref{eqCDndintro} implies that the Ricci curvature of the rescaled Riemannian metric $d\theta(\cdot,J\cdot)+\lambda^{-2}\theta^2$ satisfies itself a lower bound for some values of the scaling parameter $\lambda$. The compactness result  we obtain (see Theorem \ref{myers}) is then a consequence of the classical Bonnet-Myers theorem on Riemannian manifolds. We believe Theorem \ref{myers} is not optimal and we actually conjecture:

\begin{conjecture}\label{conjecture}
If $\rho_1-\frac{\kappa\sqrt{\rho_3}}{\sqrt{\rho_2}} > 0$, then the manifold $\M$ is compact.
\end{conjecture}

The conjecture is strongly supported by the fact that we prove in Section \ref{GB} that if $\rho_1-\frac{\kappa\sqrt{\rho_3}}{\sqrt{\rho_2}} > 0$, then the volume of $\M$ is finite and the sub-Laplacian has a spectral gap, which for instance  proves the conjecture in the case where $\M$ is a Lie group. We can observe that  in the case of Sasakian manifolds, for which $\rho_3=0$ and $\rho_1$ is precisely a lower bound for the Ricci curvature of the Webster-Tanaka connection, the conjecture has been proved  in \cite{BG1}.  

\

In Section \ref{GB} we prove  that if $\rho_1-\frac{\kappa\sqrt{\rho_3}}{\sqrt{\rho_2}} > 0$, then the operator $-L$ has a spectral gap of size
\[
\lambda_1 \ge \frac{\rho_1\rho_2-\kappa\sqrt{\rho_2\rho_3}}{\rho_2+\kappa}.
\]
This result can be seen as a Lichnerowicz type estimate on contact Riemannian manifolds. We should mention that such estimates have already been obtained on CR manifolds (see for instance (\cite{B}, \cite{Chiu}, \cite{G}, \cite{LL}) and that our result is not sharp since the lower bound does not involve the dimension of the manifold (in the Riemannian case our bound writes $\lambda_1 \ge \rho$ where $\rho $ is a lower bound on the Ricci curvature). However, the main point here, is that we do not assume the compactness of the manifold. Therefore the methods of the above mentioned papers which consist to integrate over the manifold Bochner's identity, can not be used in our framework. Instead, we need to use heat semigroup methods which are perfectly adapted to "integrate" Bochner's identity in non compact frameworks.

To conclude, let us stress that the existence of a spectral gap does not imply compactness of $\M$ but  is clearly a first important step toward the proof of our conjecture \ref{conjecture}. The missing ingredient here is an ultracontractivity property of the heat semigroup. More precisely, if we could establish that, under the condition $\rho_1-\frac{\kappa\sqrt{\rho_3}}{\sqrt{\rho_2}} >0$, the heat kernel $p_t(x,y)$ satisfies for some constants $C>0$ and $D >0$,  the global  small time estimate 
\[
p_t(x,y)\le \frac{C}{t^{D/2}}, \quad x,y \in \M, \quad  0<t \le 1,
\]
then we would have a Sobolev inequality, that could then be improved into a tight Sobolev inequality by using the existence of a spectral gap for $L$. It is then known by using the Moser's iteration technique that a tight Sobolev inequality implies the compactness of the underlying space. This strategy is essentially the one that proved to be succesful in the Sasakian case (see \cite{BG1} for more details) and we hope to adapt it to the present framework in a future work.

\

\textbf{Acknowledgments:} The authors thank an anonymous referee for pointing out several references.
 
\section{The sub-Laplacian of a  contact Riemannian manifold}\label{contact}

Let $(\M,\theta)$ be a $2n+1$-dimensional smooth contact manifold. On $\M$ there is a unique smooth vector field $Z$, the Reeb vector field, that satisfies
\[
\theta(Z)=1,\quad \mathcal{L}_Z(\theta)=0,
\]
where $\mathcal{L}_Z$ denotes the Lie derivative with respect to  $Z$. The kernel of $\theta$ defines a  $2n$ dimensional subbundle  of $\M$ which shall be referred to as the set of horizontal directions and denoted  $\mathcal{H}(\bM)$. The vector field $Z$ is transverse to  $\mathcal{H}(\bM)$ and will be referred to as the vertical direction.

According to \cite{T},  it is always possible to find a Riemannian metric $g$ and a $(1,1)$-tensor field $J$ on $\M$ so that for every  vector fields $X, Y$
\[
g(X,Z)=\theta(X),\quad J^2=-I+\theta\otimes Z, \quad g(X,JY)=(d\theta)(X,Y).
\]
The triple $(\M, \theta,g)$ is called a contact Riemannian manifold, a geometric structure well studied by Tanno in \cite{T}.  On a contact Riemannian manifold, the Riemannian structure of $\M$ is actually often confined to  the background whereas  the sub-Riemannian structure of $\M$ carries more fundamental informations about the contact structure (see \cite{BD}, \cite{R} \cite{T}).

If $f:\M \to \mathbb{R}$ is a smooth function, we denote by $\nabla_\mathcal{H} f$ the horizontal gradient of $f$ which is defined as the projection of the Riemannian gradient of $f$ onto the horizontal space $\mathcal{H}(\M)$. The sub-Laplacian $L$ of the contact Riemannian manifold $(\M, \theta,g)$ is then defined as the generator of the symmetric Dirichlet form
\[
\mathcal{E}(f,g) =\int_\M \langle \nabla_\mathcal{H} f , \nabla_\mathcal{H} g \rangle d\mu,
\]
where $\mu$ is the Borel measure given the $2n+1$ volume form $\theta \wedge (d\theta)^n$. The diffusion operator $L$ is not elliptic but subelliptic of order $1/2$ (see \cite{BD}). We can observe that, as a direct consequence of the definition of $L$, we have
\[
L=\Delta-Z^2,
\]
where $\Delta$ is the Laplace-Beltrami of the Riemannian structure $(\M,g)$. The following lemma will be useful:

\begin{lemma}\label{ess}
If the Riemannian manifold $(\M,g)$ is complete, then $L$ is essentially self-adjoint on the space $C^\infty_0(\M)$ of smooth and compactly supported functions..
\end{lemma}

\begin{proof}
If $(\M,g)$ is complete, then from  \cite{strichartz1}, there exists a sequence $h_n$ in $C^\infty_0(\M)$ such that $\| \nabla_\mathcal{H} f \|_\infty + \| Zf \|_\infty \to 0$ when $n \to \infty$. In particular $\| \nabla_\mathcal{H} f \|_\infty \to 0$, and thus from \cite{Strichartz},  $L$ is essentially self-adjoint on the space $C^\infty_0(\M)$.
\end{proof}

In the sequel of the paper we always assume that $(\M,g)$ is complete.

\

We denote by $\nabla^R$ the Levi-Civita connection on $\M$. The following $(1,2)$ tensor field $Q$ on $(\M,g)$ that was introduced by Tanno in \cite{T} as  follows:
\[
Q(X,Y)=(\nabla^R_{X}J)Y+[(\nabla^R_{Y}\theta)JX]Z+\theta(X)J(\nabla^R_{Y}Z)
\]
will play a pervasive role in this paper. A fundamental result due to Tanno is that $(\M, \theta, J|_{\mathcal{H}(\M)})$ is a strongly pseudo convex CR manifold if and only if $Q=0$

 Besides the Riemannian connection $\nabla^R$, there is a canonical sub-Riemannian connection that was introduced by Tanno in \cite{T} and which generalizes the Tanaka-Webster connection of the CR manifolds. This connection denoted by  $\nabla$ in the sequel, is much more naturally associated with the study of the sub-Laplacian $L$. In terms of the Riemannian connection, the Tanno's connection writes for every vector fields $X,Y$,
\[
\nabla_XY=\nabla^R_XY+\theta(X)JY-\theta(Y)\nabla^R_XZ+[(\nabla^R_X\theta)Y]Z.
\]
This connection $\nabla$ is more intrinsically characterized as follows:

\begin{proposition}[S. Tanno, \cite{T}]\label{tanno} The  connection $\nabla$ on $(\M,\theta,g)$ is the unique linear connection that satisfies:
\begin{enumerate}
\item $\nabla\theta=0$;
\item $\nabla Z=0$;
\item $\nabla g=0$;
\item ${T}(X,Y)=d\theta(X,Y)Z$ for any $X,Y\in\mathcal{H}(\M)$;
\item ${T}(Z,JX)=-J{T}(Z,X)$ for any vector field $X$;
\item $(\nabla_XJ)Y=Q(Y,X)$ for any vector fields $X,Y$.
\end{enumerate}
where ${T}(\cdot,\cdot)$ is the torsion tensor with respect to $\nabla$.
\end{proposition}


If $X$ is a horizontal vector field, so is $T(Z,X)$. As a consequence if we define $\tau(X)=T(Z,X)$, $\tau$ is a symmetric horizontal endomorphism which satisfies $\tau \circ  J+J \circ  \tau=0$. In the context of CR manifolds, $\tau$ is referred to as the pseudo-Hermitian torsion. We can observe that $\tau=0$ is equivalent to the fact that the contact structure is of K type (see \cite{T}).

\

For our purpose, it will be expedient to work in local frames that are adapted to the contact structure. If $X_1,X_2,\cdots,X_{2n}$ is a  local orthonormal frame of $\mathcal{H}(\M)$, all the local geometry of the contact manifold is contained into the structure coefficients that are defined by 
\begin{equation}\label{eq-structure}
[X_i,X_j]=\sum_{k=1}^{2n}w_{ij}^kX_k+\gamma_{ij}Z,\quad [X_i,Z]=\sum_{j=1}^{2n}\delta_{i}^jX_j
\end{equation}
where $w_{ij}^k$, $\gamma_{ij}$, $\delta_{i}^j$ are smooth functions. It is easy to see that
\begin{equation}\label{relations}
w_{ij}^k=-w_{ji}^k, \quad \gamma_{ij}=-\gamma_{ji},\quad i,j,k=1,\cdots 2n.
\end{equation}

In the local  frame $\{X_1,\cdots,X_{2n},Z\}$ as above,  the sub-Laplacian $L$ can be written
\[
L=-\sum_{i=1}^{2n}X_i^*X_i,
\] 
where $X_i^*$ is the formal adjoint of $X_i$ with respect to the volume measure $\mu$. From \eqref{eq-structure}, we obtain
\[
X_i^*=-X_i+\sum_{k=1}^{2n}w_{ik}^k.
\]
Hence, we can write locally
\[
L=\sum_{i=1}^{2n}X_i^2+X_0,
\]
where 
\begin{equation}\label{eq-X0}
X_0=-\sum_{i,k=1}^{2n}w_{ik}^kX_i.
\end{equation}

By \eqref{eq-structure}, one can then easily calculate the Christoffel's symbols of the sub-Riemannian connection:
\[
 \nabla_{X_i}X_j=\sum_{k=1}^{2n}\Gamma_{ij}^kX_k,\quad \nabla_ZX_i=\frac{1}{2}\sum_{k=1}^{2n}(\delta_k^i-\delta_i^k)X_k
\]
where $\Gamma_{ij}^k=\frac{1}{2}(w_{ij}^k+w_{ki}^j+w_{kj}^i)$. It is also easy to see that
\[
\tau(X_i)=\frac{1}{2}\sum_{k=1}^{2n}(\delta_k^i+\delta_i^k)X_k, \quad{T}(X_j,X_k)=-\gamma_{jk}Z\quad\mbox{and}\quad JX_i=\sum_{j=1}^{2n}\gamma_{ij}X_j.
\]

In the case of CR Sasakian manifolds, in addition to the relations in \eqref{relations},  we also have the skew-symmetry of the $\delta_i^j$s, i.e., $\delta_i^j=-\delta_j^i$  for all $i,j=1,\cdots, 2n$, which implies that the torsion $\tau$ vanishes (see \cite{BG1}).

In our general  case, though the skew-symmetry is no more satisfied, we can still always find a basis such that the diagonal entries of  $\tau$ vanish. i.e., $\delta_i^i=0$, for all $i=1,\cdots,2n$. Indeed, let $\lambda$ be an eigenvalue of $\tau$ and $X$ a corresponding eigenvector.  Since $\tau \circ J+J\circ\tau=0$, this implies that $-\lambda$ is also an eigenvalue of $\tau$. Hence $\tau$ is similar to the diagonal matrix
\[
A=\begin{pmatrix}A_1 & & \\
 & \cdots & \\
 & & A_n
\end{pmatrix},
\quad A_i=\begin{pmatrix}\lambda_1 & 0 \\
 0 & -\lambda_1
\end{pmatrix},
\quad
i=1,\cdots, n. 
\]
Since we have 
$A_i\sim\begin{pmatrix}0 & \lambda_i  \\
  \lambda_i& 0
\end{pmatrix}:= \tilde{A}_i$, thus $A\sim \begin{pmatrix}\tilde{A}_1 & & \\
 & \cdots & \\
 & & \tilde{A}_n
\end{pmatrix}$.

In the sequel, we thus always choose the local frame such that $\delta_i^i=0$, $i=1,\cdots, 2n$.

\section{The generalized curvature dimension inequality}\label{Bochners}

\subsection{Bochner's formulas}

Our first goal will be to work out the Bochner's type formulas for the sub-Laplacian $L$. We follow the methods of \cite{BG1} and use the $\Gamma_2$ formalism introduced in \cite{Bakry-Emery}.

Let us consider the first order differential bilinear form:
\[
\Gamma(f,g)=\frac{1}{2}(L(fg)-fLg-gLf), \quad f,g\in C^\infty(\M),
\]
and observe that 
\[
\Gamma(f,g)=\langle\nabla_{\mathcal{H}}f,\nabla_{\mathcal{H}}g\rangle,
\]
where $\nabla_{\mathcal{H}}$ is the horizontal gradient. $\Gamma(f)=\Gamma(f,f)$ is known as  \textit{le carr\'e du champ}.
Similarly we define for every $f,g\in C^\infty(\M)$,
\[
\Gamma^Z(f,g)=\langle\nabla_{\mathcal{V}}f,\nabla_{\mathcal{V}}g \rangle,
\]
where $\nabla_{\mathcal{V}}$ is the vertical gradient of $\M$. We also introduce the second order differential bilinear forms:
\begin{equation}\label{eqgamma2}
\Gamma_2(f,g)=\frac{1}{2}(L\Gamma(f,g)-\Gamma(f,Lg)-\Gamma(g,Lf))
\end{equation}
and
\begin{equation}\label{gammaT}
\Gamma_2^Z(f,g)=\frac{1}{2}(L\Gamma^Z(f,g)-\Gamma^Z(f,Lg)-\Gamma^Z(g,Lf)).
\end{equation}

Throughout the Section, we work in a local frame that satisfies
\[
[X_i,X_j]=\sum_{k=1}^{2n}w_{ij}^kX_k+\gamma_{ij}Z,\quad [X_i,Z]=\sum_{j=1}^{2n}\delta_{i}^jX_j
\]
with $\delta_i^i=0$.

The following tensorial quantity will play a crucial role in our discussion.

\begin{definition}
Let $\mathbf{Ric}(\cdot,\cdot)$ and ${T}(\cdot,\cdot)$ respectively denote the Ricci and torsion tensors of the sub-Riemannian connection $\nabla$. For $f\in C^\infty(\M)$ we define:
\begin{align}\label{eqR}
 & \mathcal{R}(f,f) \\
\notag =&   \mathbf{Ric}(\nabla_{\mathcal{H}}f,\nabla_{\mathcal{H}}f)+\frac{n}{2}\| \nabla_\mathcal{V} f \|^2  -\sum_{l,k=1}^{2n}\left( \bigg((\nabla_{X_l}{T})(X_l,X_k)f (X_kf)\bigg)+{T}(X_l,{T}(X_l,X_k))fX_kf\right).
\end{align}
\end{definition}
From its definition, it is obvious that $\mathcal{R}$ is an intrinsic first order differential bilinear form on $\M$. The following proposition provides its computations in terms of the structure constants of the local frame.

\begin{lemma}
We have:
\begin{align*}
\mathcal{R} (f,f) =  \sum_{k,l=1}^{2n}\mathcal{R}_{kl} X_kfX_lf + \sum_{k=1}^{2n}\left(\sum_{l,j=1}^{2n}w_{jl}^l\gamma_{kj}+\sum_{1\leq l<j \leq 2n}w_{lj}^k\gamma_{lj}-\sum_{j=1}^{2n}X_j\gamma_{kj} \right)ZfX_kf+\frac{n}{2} (Zf)^2 ,
\end{align*}
with 
\[
 \mathcal{R}_{kl} =\sum_{j=1}^{2n}\gamma_{kj}\delta_j^l+\sum_{j=1}^{2n}(X_lw_{kj}^j-X_jw_{lj}^k) +\sum_{i,j=1}^{2n}w_{ji}^i w_{kj}^l-\sum_{j=1}^{2n}w_{ki}^iw_{li}^i+\frac{1}{2}\sum_{1\leq i<j\leq 2n}\left(w_{ij}^l w_{ij}^k-(w_{lj}^i+w_{li}^j)(w_{kj}^i+w_{ki}^j)\right).
\]
\end{lemma}

\begin{proof}
We write $\mathcal{R}(f,f)$ as follows
\[
\mathcal{R}(f,f)=\mathcal{R}_I(f,f)+\mathcal{R}_{II}(f,f)+\mathcal{R}_{III}(f,f),
\]
where 
\begin{eqnarray}
\mathcal{R}_I(f,f) &=& \sum_{l,k=1}^{2n}\bigg( \mathbf{Ric}(X_l,X_k)X_lfX_kf+{T}(X_l,{T}(X_l,X_k))fX_kf\bigg), \label{RI}\\
\mathcal{R}_{II}(f,f) &=& -\sum_{k,l=1}^{2n} \bigg((\nabla_{X_l}{T})(X_l,X_k)f (X_kf)\bigg), \\
\mathcal{R}_{III}(f,f) &=&\frac{n}{2} (Zf)^2.
\end{eqnarray}
Straightforward but tedious calculations show that
\begin{eqnarray}
\sum_{l,k=1}^{2n}{T}(X_l,{T}(X_l,X_k))fX_kf&=& -\sum_{l,k,j=1}^{2n}\frac{\delta_k^j+\delta_j^k}{2}\gamma_{jl}X_lfX_kf \label{tau2}\\
\sum_{l,k=1}^{2n}\mathbf{Ric}(X_l,X_k)X_lfX_kf&=& \sum_{i,j,l,k=1}^{2n} \bigg(\Gamma_{lk}^i\Gamma_{ji}^j-\Gamma_{jk}^i\Gamma_{li}^j
-w_{jl}^i \Gamma_{ik}^j\bigg)X_lfX_kf \nonumber \\
&+&\sum_{j,l,k=1}^{2n}\bigg( (X_j\Gamma_{lk}^j)-(X_l\Gamma_{jk}^j)\bigg)X_lfX_kf
-\sum_{l,k,j=1}^{2n}\gamma_{jl} \frac{\delta_j^k-\delta_k^j}{2} X_lfX_kf, \nonumber
\end{eqnarray}
which implies that
\begin{eqnarray*}
\mathcal{R}_I(f,f) &=&  \sum_{k,l=1}^{2n}\mathcal{R}_{kl} X_kfX_lf .
\end{eqnarray*}
We also calculate in a direct way that
\begin{eqnarray*}
\mathcal{R}_{II}(f,f) &=& \sum_{k=1}^{2n}\left(\sum_{l,j=1}^{2n}w_{jl}^l\gamma_{kj}+\sum_{1\leq l<j \leq 2n}w_{lj}^k\gamma_{lj}-\sum_{j=1}^{2n}X_j\gamma_{kj} \right)ZfX_kf.
\end{eqnarray*}
By combining the above terms we have the lemma.
\end{proof}

With these preliminary results in hands, we can now turn to the proof of the horizontal Bochner's formula:
\begin{theorem}\label{T:bochner}
For every $f\in C^\infty(\bM)$, the following {\bf{Horizontal Bochner formula}} holds:
\begin{align}\label{bochner}
\Gamma_{2}(f)=\| \nabla_\mathcal{H}^2 f \|^2 +\mathcal{R}(f,f)-2\sum_{i,j=1}^{2n}\gamma_{ij}(X_jZf)(X_if).
\end{align}
\end{theorem}

\begin{proof}
Our method is close to the method used in \cite{BG1} to obtain a horizontal Bochner's formula, so we refer to this paper for further details and only give the mains steps in the calculations.

It is of course enough to prove \eqref{bochner} in the local frame $\{X_1,...,X_{2n},Z\}$. Observe that 
\[
X_i X_j f = f_{,ij} + \frac 12 [X_i,X_j]f,
\]
where we have let
\begin{equation}\label{symmder}
f_{,ij} = \frac 12 (X_i X_j + X_j X_i) f.
\end{equation}
Using \eqref{eq-structure}, we find 
\begin{equation}\label{nonsym}
X_iX_jf = f_{,ij} + \frac{1}{2} \sum_{\ee=1}^{2n} \omega^\ee_{ij} X_\ee
f + \frac{1}{2} \gamma_{ij} Z f.
\end{equation}
Now, starting from the definition \eqref{eqgamma2} of $\Gamma_2(f)$, we obtain
\begin{align*}
\Gamma_2(f)  =  \sum_{i=1}^{2n} X_i f [X_0,X_i]f - 2
\sum_{i,j=1}^{2n} X_if [X_i,X_j]X_j f 
 + \sum_{i,j=1}^{2n} X_i f
[[X_i,X_j],X_j]f  +  \sum_{i,j=1}^{2n} (X_jX_i f)^2,
\end{align*}
where $X_0$ is defined by \eqref{eq-X0}. From \eqref{nonsym} we have
\begin{align*}
\sum_{i,j=1}^{2n} (X_jX_i f)^2  & =  \sum_{i,j=1}^{2n} f_{,ij}^2 + \frac{1}{2} \sum_{1\le
i<j\le 2n}\left(\sum_{\ee=1}^{2n} \omega^\ee_{ij} X_\ee f\right)^2 +
\frac{1}{2} \sum_{1\le i<j\le 2n}\left(
\gamma_{ij}Z f\right)^2 \\
& + \sum_{1\le i<j\le 2n}\sum_{\ee=1}^{2n} 
\omega^\ee_{ij} \gamma_{ij} Zf X_\ee f,
\notag
\end{align*}
and therefore, 
\begin{align}\label{bochner1}
\Gamma_2(f) & =  \sum_{i,j=1}^{2n} f_{,ij}^2- 2
\sum_{i,j=1}^{2n} X_if [X_i,X_j]X_j f + \sum_{i,j=1}^{2n} X_i f
[[X_i,X_j],X_j]f 
\\
& + \sum_{i=1}^{2n} X_i f [X_0,X_i]f  + \frac{1}{2} \sum_{1\le
i<j\le 2n}\left(\sum_{\ee=1}^{2n} \omega^\ee_{ij} X_\ee f\right)^2 +
\frac{1}{2} \sum_{1\le i<j\le 2n}\left(
\gamma_{ij} Zf\right)^2 
\notag\\
& + \sum_{1\le i<j\le 2n}\sum_{\ee=1}^{2n} 
\omega^\ee_{ij} \gamma_{ij} Zf X_\ee f. 
\notag
\end{align}
By plugging in \eqref{eq-structure} and completing the square, we obtain 
\begin{align*}
\Gamma_2(f) & = \sum_{\ee=1}^{2n} \left( f_{,\ell\ell} -\sum_{i=1}^{2n}
\omega_{i\ee}^\ee X_i f \right)^2 +  2 \sum_{1 \le \ee<j \le 2n}
\left( f_{,j\ell} -\sum_{i=1}^{2n} \frac{\omega_{ij}^\ee
+\omega_{i\ee}^j}{2} X_i f \right)^2
\\
& - 2 \sum_{i,j=1}^{2n} \gamma_{ij} X_j Zf\
X_i f + \mathcal R(f),
\end{align*}
where we used the fact that $ \sum_{1\le i<j\le 2n}\left(\gamma_{ij} Zf\right)^2=n (Zf)^2$. At last, we complete the proof of \eqref{bochner} by realizing that the square of the Hilbert-Schmidt norm of the horizontal Hessian $\nabla_\mathcal{H}^2 f$ is given by
\begin{align}\label{hessian}
\| \nabla_\mathcal{H}^2 f \|^2 & = \sum_{\ee=1}^{2n} \left( f_{,\ell\ell} -\sum_{i=1}^{2n}
\omega_{i\ee}^\ee X_i f \right)^2 +  2 \sum_{1 \le \ee<j \le 2n}
\left( f_{,j\ell} -\sum_{i=1}^{2n} \frac{\omega_{ij}^\ee
+\omega_{i\ee}^j}{2} X_i f \right)^2.
\end{align}
\end{proof}

Our next goal is  to derive a vertical Bochner's formula. We first give the formula in terms of the structure constants and will provide the tensorial expressions afterwards.

\begin{theorem}\label{T:vertical}
For every $f\in C^\infty(\bM)$, 
\begin{equation}\label{vertical}
\Gamma^Z_2(f)=\sum_{i=1}^{2n}(X_iZf)^2+\frac{1}{2}\sum_{i,l=1}^{2n}(\delta^l_i+\delta^i_l)(X_iX_lf+X_lX_if)Zf+\sum_{i,l=1}^{2n}\left(X_i\delta_i^l-\sum_{k=1}^{2n}w_{ik}^k\delta_i^l+\sum_{k=1}^{2n}Zw_{lk}^k \right)X_lfZf.
\end{equation}
\end{theorem}

\begin{proof}
From \eqref{gammaT}, we know that
\begin{equation}\label{gamma-2-T}
\Gamma_2^Z(f)=\Gamma(Zf)+[L,Z]fZf.
\end{equation}
Moreover, since
\begin{eqnarray*}
[L,Z]f&=&[X_0,Z]f+\sum_{i=1}^{2n}(X_i[X_i,Z]f+[X_i,Z]X_if) 
\end{eqnarray*}
we can easily compute that
\begin{eqnarray}\label{LT}
[L,Z]f=-\sum_{i,k,l=1}^{2n}w_{ik}^k\delta_i^lX_lf+\sum_{l,k=1}^{2n}(Zw_{lk}^k)X_lf+\sum_{i,l=1}^{2n}(X_i\delta_i^l)X_lf+\frac{1}{2}\sum_{i,l=1}^{2n}(\delta_i^l+\delta_l^i)(X_iX_l+X_lX_i)f. \quad
\end{eqnarray}
Plug this expression back in (\ref{gamma-2-T}), we have the expression for $\Gamma_2^Z(f)$.
\end{proof}

To stress that the formula, of course does not depend on the local frame, we can rewrite it as follows:
\begin{theorem}
For any smooth function $f\in C^\infty(\M)$, we have 
\begin{align*}
\Gamma^Z_2(f)= & \|\nabla_{\mathcal{H}}\nabla_{\mathcal{V}}f\|^2+\mathbf{Ric}(\nabla_{\mathcal{H}}f,\nabla_{\mathcal{V}}f) \\
 & +2\sum_{i,l=1}^{2n}\tau(X_i)X_ifZf+\sum_{k=1}^{2n}(\nabla_{X_k}{T})(Z,X_k)fZf -2\sum_{k=1}^{2n}\nabla_{\tau(X_k)}X_kfZf
\end{align*}
\end{theorem}

\begin{proof}
Since
\begin{eqnarray*}
(\nabla_{X_k}{T})(Z,X_i)&=&\nabla_{X_k}(\tau(X_i))-\tau(\nabla_{X_k}X_i),
\end{eqnarray*}
we have that
\begin{eqnarray*}
(\nabla_{X_k}{T})(Z,X_k)&=&\frac{1}{2}\sum_{l=1}^{2n}X_k(\delta_k^l+\delta_l^k)X_l+\frac{1}{2}\sum_{l,j=1}^{2n}(\delta_k^l+\delta_l^k)\Gamma_{kl}^jX_j-\frac{1}{2}\sum_{l,j=1}^{2n}w_{lk}^k\left(\delta_l^j+\delta_j^l\right)X_j.
\end{eqnarray*}
and simple calculations give us
\begin{align*}
\mathbf{Ric}(Z,X_i)=& \frac{1}{2}\sum_{j=1}^{2n}X_j(\delta_j^i-\delta_i^j)+\frac{1}{2}\sum_{j,k=1}^{2n}w_{jk}^j(\delta_k^i-\delta_i^k)\\
 & -\sum_{j=1}^{2n}Zw_{ji}^j-\frac{1}{2}\sum_{k,j=1}^{2n}\Gamma_{ji}^k(\delta_j^k-\delta_k^j)-\sum_{k,j=1}^{2n}\delta_j^k\Gamma_{ki}^j.
\end{align*}
As a consequence, we obtain
\begin{eqnarray*}
&&\sum_{i=1}^{2n}\mathbf{Ric}(Z,X_i)X_ifZf+\sum_{k=1}^{2n}(\nabla_{X_k}{T})(Z,X_k)fZf\\
&=&\sum_{i,j=1}^{2n}X_j\delta_j^iX_ifZf
+\sum_{i,j,k=1}^{2n}w_{jk}^j\delta_k^iX_ifZf
-\sum_{i,j=1}^{2n}(Zw_{ji}^j)X_ifZf\\
&+&\left(\frac{1}{2}\sum_{i,j,k=1}^{2n}(\delta_k^j+\delta_j^k)\Gamma_{kj}^iX_if
-\frac{1}{2}\sum_{i,j,k=1}^{2n}\Gamma_{ji}^k(\delta_j^k-\delta_k^j)X_if
-\sum_{i,j,k=1}^{2n}\delta_j^k\Gamma_{ki}^jX_if\right)Zf.
\end{eqnarray*}
By taking into account
\[
\Gamma_{ji}^k=\Gamma_{kj}^i-(w_{ij}^k+w_{ik}^j)=w_{kj}^i-\Gamma_{kj}^i,\qquad \Gamma_{ki}^j=-\Gamma_{kj}^i, 
\]
we have that
\[
\frac{1}{2}\sum_{i,j,k=1}^{2n}(\delta_k^j+\delta_j^k)\Gamma_{kj}^iX_if
-\frac{1}{2}\sum_{i,j,k=1}^{2n}\Gamma_{ji}^k(\delta_j^k-\delta_k^j)X_if
-\sum_{i,j,k=1}^{2n}\delta_j^k\Gamma_{ki}^jX_if=\frac{1}{2}\sum_{i,j,k=1}^{2n}(w_{ij}^k+w_{ik}^j)(\delta_k^j+\delta_j^k)X_if. 
\]
Moreover, notice that
\begin{equation}\label{wd}
\frac{1}{2}\sum_{i,j,k=1}^{2n}\left(\delta^j_k+\delta^k_j\right)\left(w_{ij}^k+w_{ik}^j\right)X_ifZf
=2\sum_{k=1}^{2n}\nabla_{\tau(X_k)}X_kfZf,
\end{equation}
so that we can write 
\begin{align} \label{first-order-term}
&\sum_{i,j=1}^{2n}X_j\delta_j^iX_ifZf
+\sum_{i,j,k=1}^{2n}w_{jk}^j\delta_k^iX_ifZf-\sum_{i,j=1}^{2n}(Zw_{ji}^j)X_ifZf& \nonumber \\ 
&=\sum_{i=1}^{2n}\mathbf{Ric}(Z,X_i)X_ifZf+\sum_{k=1}^{2n}(\nabla_{X_k}{T})(Z,X_k)fZf 
-2\sum_{k=1}^{2n}\nabla_{\tau(X_k)}X_kfZf.
\end{align}
At the end we conclude the proposition by comparing with the expression in \eqref{vertical}.
\end{proof}

\subsection{Generalized curvature dimension bounds}

With the two Bochner's formulas in hands, we are now ready to give the suitable curvature dimension conditions on contact manifolds.  To this purpose, we introduce the relevant geometric quantities. As in the previous subsection, we work in a local frame.

\

 The vector field
\[
V=\sum_{i=1}^{2n}\mathbf{Ric}(Z,X_i)X_i+(\nabla_{X_i}{T})(Z,X_i),
\]
obviously does not depend on the choice of the local frame and is therefore an intrinsic invariant of the manifold. In terms of the structure constants, we compute
\[
V=\sum_{i,j,l=1}^{2n}\left(\frac{\delta^l_j+\delta^j_l}{2}\right)\left(w_{il}^j+w_{ij}^l\right)X_i+\sum_{i=1}^{2n}
\left(\sum_{j=1}^{2n}X_j\delta_j^i-\sum_{j,k=1}^{2n}w_{jk}^k\delta_j^i+\sum_{k=1}^{2n}Zw_{ik}^k\right)X_i.
\]
We then consider the first-order quadratic  differential form defined for   $f \in C^\infty(\M)$ by 
\[
\tau_2(f)=\sum_{l,k=1}^{2n}{T}(X_l,{T}(X_l,X_k))fX_kf,
\]
and the horizontal trace of the Tanno tensor $Q$ which is the vector field given by  $\mathbf{Tr}_{\mathcal{H}}Q:=\sum_{l=1}^{2n}Q(X_l,X_l)=\sum_{l=1}^{2n}(\nabla_{X_l}J)X_l$ . Our main result is the following:

\begin{theorem}\label{CD-geo-thm}
 Assume  there exist constants $c_1\in\R$, $c_2\geq0$, $c_3\geq0$ and $\iota\geq0$ such that for every $f\in C^\infty(\M)$,
\begin{align}\label{geometry}
\mathbf{Ric} (\nabla_\mathcal{H} f) +\tau_2 (f)  \geq c_1 \|\nabla_\mathcal{H} f\|^2 ,
\| (\Q )f \|^2\leq c_2\| \nabla_\mathcal{H} f\|^2,
\end{align}
\begin{align*}
\|Vf \|^2\leq c_3 \|\nabla_\mathcal{H} f\|^2 , \quad
\|\tau (\nabla_\mathcal{H} f) \|^2\leq\iota \| \nabla_\mathcal{H} f\|^2.
\end{align*}
Then for all $\nu>0$ and $f \in C^\infty(\M)$,
\begin{equation*}
\Gamma_2(f)+\nu\Gamma^Z_2(f)\geq
\frac{1}{2n}(Lf)^2+
\left(c_1-\frac{1}{\nu}\right)\Gamma(f)-\left(c_2+c_3\nu\right)\sqrt{\Gamma(f)\Gamma^Z(f)}+\left(\frac{n}{2}-\frac{\iota}{4}\nu^2\right)\Gamma^Z(f).
\end{equation*}  
\end{theorem}

\begin{proof}
To derive the generalized curvature-dimension inequality, let us first introduce the first-order differential forms $\mathcal{U}$ and $\mathcal{T}$ in the local frame $\{X_1,\cdots, X_{2n}\}$ such that
\begin{equation}\label{rmTU}
\mathcal{T}(f,f)=\sum_{k=1}^{2n}\|{T}(X_k,\nabla_{\mathcal{H}}f)^2\|,\quad \mathcal{U}(f,f)=\sum_{k=1}^{2n}\|\tau(X_k)\|^2(Zf)^2.
\end{equation}
A simple computation shows that 
\begin{equation}
\label{U}
\mathcal{U}(f,f)=\sum_{j,l=1}^{2n}\left(\frac{\delta^l_j+\delta^j_l}{2}\right)^2\left(Zf\right)^2,\quad 
\mathcal{T}(f,f)=\sum_{j=1}^{2n}\left(\sum_{i=1}^{2n}\gamma_{ij}X_if\right)^2.
\end{equation}

Let us also consider $\mathcal{S}(f)=VfZf$ so  that
\begin{equation}\label{rmS}
\mathcal{S}(f)=\mathbf{Ric}(\nabla_{\mathcal{V}}f,\nabla_{\mathcal{H}}f)+\sum_{i=1}^{2n}(\nabla_{X_i}{T})(Z,X_i)fZf.
\end{equation}
From \eqref{bochner} and \eqref{vertical}, by using the fact that $\delta_i^i=0$, we have that
\begin{eqnarray*}
\Gamma_2(f,f)+\nu\Gamma^Z_2(f,f)
&=&\sum_{l=1}^{2n}\left(X_l^2f-\sum_{i=1}^{2n}w_{il}^lX_if\right)^2-2\sum_{i,j=1}^{2n}\gamma_{ij}(X_jZf)(X_if)\\
&+&\nu\sum_{i=1}^{2n}(X_iZf)^2+2\nu\sum_{1\leq l<j\leq 2n}\left(\frac{\delta^l_j+\delta^j_l}{2}\right)\left(\frac{X_jX_l+X_lX_j}{2}\right)fZf\\
&+&2\sum_{1\leq l<j\leq 2n}\left(\left(\frac{X_lX_j+X_jX_l}{2}\right)f-\sum_{i=1}^{2n}\left(\frac{w_{il}^j+w_{ij}^l}{2}\right)X_if\right)^2  \\
&+&\nu\sum_{l=1}^{2n}\left(\sum_{i=1}^{2n}X_i\delta_i^l-\sum_{i,k=1}^{2n}w_{ik}^k\delta_i^l+\sum_{k=1}^{2n}Zw_{lk}^k\right)X_lf Zf+\mathcal{R}(f,f).
\end{eqnarray*}
We write the above equation as follows:
\[
\Gamma_2(f,f)+\nu\Gamma^Z_2(f,f)=\mathcal{B}_I+\mathcal{B}_{II}+\mathcal{B}_{III}+\nu\sum_{l=1}^{2n}\left(\sum_{i=1}^{2n}X_i\delta_i^l-\sum_{i,k=1}^{2n}w_{ik}^k\delta_i^l+\sum_{k=1}^{2n}Zw_{lk}^k\right)X_lf Zf+\mathcal{R}(f,f),
\]
where 
\begin{eqnarray*}
& &\mathcal{B}_I = \sum_{l=1}^{2n}\left(X_l^2f-\sum_{i=1}^{2n}w_{il}^lX_if\right)^2,\\
& &\mathcal{B}_{II} = -2\sum_{i,j=1}^{2n}\gamma_{ij}(X_jZf)(X_if)+\nu\sum_{i=1}^{2n}(X_iZf)^2, \\
& &\mathcal{B}_{III} = 2\sum_{1\leq l<j\leq 2n}\left(\left(\frac{X_lX_j+X_jX_l}{2}\right)f-\sum_{i=1}^{2n}\left(\frac{w_{il}^j+w_{ij}^l}{2}\right)X_if\right)^2\\
& &\qquad\qquad+2\nu\sum_{1\leq l<j\leq 2n}\left(\frac{\delta^l_j+\delta^j_l}{2}\right)\left(\frac{X_jX_l+X_lX_j}{2}\right)fZf
\end{eqnarray*}

Hence from Cauchy-Schwartz inequality we obtain
\begin{eqnarray*}
\mathcal{B}_I \geq \frac{1}{2n}(Lf)^2.
\end{eqnarray*}
Also we can easily see that
\[
\mathcal{B}_{II}
\geq-\frac{1}{\nu}\sum_{j=1}^{2n}\left(\sum_{i=1}^{2n}\gamma_{ij}X_if\right)^2,
\]
and
\begin{align*}
\mathcal{B}_{III}
\geq
2\nu\sum_{1\leq l<j\leq 2n}\sum_{i=1}^{2n}\left(\frac{\delta^l_j+\delta^j_l}{2}\right)\left(\frac{w_{il}^j+w_{ij}^l}{2}\right)X_ifZf-\frac{\nu^2}{2}\sum_{1\leq l<j\leq 2n}\left(\left(\frac{\delta^l_j+\delta^j_l}{2}\right)Zf\right)^2.
\end{align*}
Hence we have
\begin{align*}
\Gamma_2(f,f)+\nu\Gamma^Z_2(f,f)
\geq
 \frac{1}{2n}(Lf)^2-\frac{\nu^2}{4}\mathcal{U}(f)+\nu\mathcal{S}(f) +\mathcal{R}(f)-\frac{1}{\nu}\mathcal{T}(f),
\end{align*}
and the conclusion easily follows from the fact that
\[
\mathcal{T}(f)=\sum_{k=1}^{2n}\langle J\nabla_{\mathcal{H}}f, X_k\rangle^2=\|J\nabla_{\mathcal{H}}f\|^2=\Gamma(f).
\]
\end{proof}
In the case of Sasakian manifolds, we have $V=0$, $Q=0$, $\tau=0$ and we recover the curvature dimension inequality introduced in \cite{BG1}. 

In view of Theorem \ref{CD-geo-thm}, it is then natural to set the following definition:

\begin{definition}
We say that $\M$ satisfies the \textbf{generalized curvature-dimension inequality} CD($\rho_1$, $\rho_2$, $\rho_3$, $\kappa$, $m$) with respect to $L$ and $\Gamma^Z$ if there exist constants $\rho_1, \rho_2\in\R$, $\rho_3>0$, $\kappa>0$, $0<m \le \infty$ such that the inequality 
\[
\Gamma_2(f)+\nu\Gamma_2^Z(f)\geq\frac{1}{m}(Lf)^2+\left(\rho_1-\frac{\kappa}{\nu}\right)\Gamma(f)+\left( \rho_2-\rho_3\nu^2\right)\Gamma^Z(f)
\]
holds for all $f\in C^\infty(\M)$ and every $\nu>0$.
\end{definition}

In particular, under the assumptions of Theorem \ref{CD-geo-thm} we easily see  that the curvature-dimension inequality $CD(\rho_1,\rho_2,\rho_3,1,2n)$ holds for every $z>0$, $w>0$, where $\rho_1=c_1-\frac{c_2z}{2}-\frac{c_3w}{2}$, $\rho_2=\frac{n}{2}-\frac{c_2}{2z}$, $\rho_3=\frac{c_3}{2w}+\frac{\iota}{4}$.   

It is very interesting to observe that Theorem \ref{CD-geo-thm} admits a partial converse.
\begin{theorem}\label{converse}
Assume that there exist constants $c_1,c_2,c_3$ and $\iota$ such that for every  $\nu>0$ and $f \in C^\infty(\M)$,
\begin{equation*}
\Gamma_2(f)+\nu\Gamma^Z_2(f)\geq
\frac{1}{2n}(Lf)^2+
\left(c_1-\frac{1}{\nu}\right)\Gamma(f)-\left(c_2+c_3\nu\right)\sqrt{\Gamma(f)\Gamma^Z(f)}+\left(\frac{n}{2}-\frac{\iota}{4}\nu^2\right)\Gamma^Z(f),
\end{equation*}  
 then, we have  for every $f \in C^\infty(\M)$,
\begin{align*}
\mathbf{Ric} (\nabla_\mathcal{H} f) +\tau_2 (f)  \geq c_1 \|\nabla_\mathcal{H} f\|^2 
\end{align*}
and
\begin{align*}
 \|\tau (\nabla_\mathcal{H} f) \|^2\leq\iota \| \nabla_\mathcal{H} f\|^2.
\end{align*}\end{theorem}
\begin{proof}
We first observe that under our assumptions the curvature-dimension inequality $CD(\rho_1,\rho_2,\rho_3,1,2n)$ holds for every $z>0$, $w>0$, where $\rho_1=c_1-\frac{c_2z}{2}-\frac{c_3w}{2}$, $\rho_2=\frac{n}{2}-\frac{c_2}{2z}$, $\rho_3=\frac{c_3}{2w}+\frac{\iota}{4}$.

For a fixed $x_0\in\M$, $u\in\mathcal{H}_{x_0}(\M)$, $v\in\mathcal{V}_{x_0}(\M)$, let $\{X_1,X_2,\cdots, X_{2n},Z \}$ be a local adapted frame around $x_0$. First we claim that for $\nu>0$, we can find a function $f\in C^\infty(\M)$ satisfying:
\begin{itemize}
\item[(i)] $\nabla_{\mathcal{H}}f(x_0)=u$,
\item[(ii)] $\nabla_{\mathcal{V}}f(x_0)=Zf(x_0)=v$,
\item[(iii)] $\left(\nabla_{\mathcal{H}}^2f(x_0)\right)_{l,j}=\frac{\nu}{2}\left(\frac{\delta^l_j+\delta^j_l}{2}\right)(x_0)v$,
\item[(iv)] $X_jZf(x_0)=\frac{1}{\nu}\sum_{i=1}^{2n}\gamma_{ij}(x_0)u_i$, for all $j=1,\cdots,2n$.
\end{itemize}
To prove this,  let $(U,\phi)$ be a local chart at $x_0$, such that $\phi(0)=x_0$ and in $U$ we have $X_j=\frac{\partial }{\partial x_j}$, $j=1,...,2n$, $Z=\frac{\partial }{\partial z}$. Then the existence of $f$ follows immediately by the existence of functions $f_1\in C^\infty(\M)$ such that 
\[
\begin{cases}
\nabla^R f_1 (x_0)=u+v,
\\
\nabla^R \nabla^Rf_1 (x_0)=0.
\end{cases}
\]
and $f_2\in C^\infty(\M)$ such that
\[
\begin{cases}
\nabla^R f_2 (x_0)=0,
\\
\left(\nabla^R\nabla^Rf_2(x_0)\right)_{l,j}=\frac{\nu}{2}\left(\frac{\delta^l_j+\delta^j_l}{2}\right)(x_0)v,
\\
X_jZ f_2 (x_0) = \frac{1}{\nu} \sum_{i=1}^{2n} \gamma_{ij}(x_0) u_i-X_jZf_1(x_0).
\end{cases}
\]
where $\nabla^R$ is the Levi-Civita connection of the Riemannian metric on $\M$. 
As in \cite{BG1}, we can easily see the existence of such  $f_1$. Also we can write $f_2$ in local coordinates $(x_1,\cdots,x_{2n},z)$ such that
\[
f_2(x,z) =\sum_{j=1}^{2n} \left(  \frac{1}{\nu} \sum_{i=1}^{2n} \gamma_{ij}(x_0) u_i-X_jZf_1(x_0) \right)  x_j z+\frac{\nu}{2}\sum_{l.j=1}^{2n}\left(\frac{\delta^l_j+\delta^j_l}{2}\right)(x_0)vx_lx_j.
\]
We then chose $f=f_1+f_2$.
Now we divide the rest of the proof into two parts. 
\begin{itemize}
\item[(1)]
First we derive the bounds for $\mathbf{Ric}(\nabla_{\mathcal{H}}f)+\tau_2(f)$. From the above claim we can find a function $f\in C^\infty(\M)$ such that (i), (ii), (iii), (iv) are satisfied with $v=0$. Moreover, by \eqref{bochner} and \eqref{vertical} we have that 
\[
\Gamma_2(f)+\nu\Gamma^Z_2(f)=\mathbf{Ric}(\nabla_{\mathcal{H}}f)+\tau_2(f)
\]
Hence we have that for all $\nu>0$, $z>0$, $w>0$,
\[
\mathbf{Ric}(\nabla_{\mathcal{H}}f)(x_0)+\tau_2(f)(x_0)\geq (\rho_1-\frac{\kappa}{\nu})\|u\|^2
\]
where $\rho_1=c_1-\frac{c_2z}{2}-\frac{c_3w}{2}$. By letting $\nu\to \infty$, $z\to 0$, $w\to 0$, we obtain that
\[
\mathbf{Ric}(\nabla_{\mathcal{H}}f)(x_0)+\tau_2(f)(x_0)\geq c_1\|u\|^2.
\]
\item[(2)] To derive the bound for $\|\tau\|^2$, we notice that the existence of the function $f\in C^\infty(\M)$ satisfying (i), (ii), (iii), (iv) implies
\begin{align*}
 & \Gamma_2(f)+\nu\Gamma^Z_2(f) \\
=& \mathcal{R}(f,f)- \frac{1}{\nu}\sum_{j=1}^{2n}\left(\sum_{i=1}^{2n}\gamma_{ij}X_if\right)^2
+\nu\sum_{i,j,l=1}^{2n}\left(\frac{\delta^l_j+\delta^j_l}{2}\right)\left(\frac{w_{il}^j+w_{ij}^l}{2}\right)X_ifZf\\
&+
\nu\sum_{l=1}^{2n}\left(\sum_{i=1}^{2n}X_i\delta_i^l-\sum_{i,k=1}^{2n}w_{ik}^k\delta_i^l+\sum_{k=1}^{2n}Zw_{lk}^k\right)X_lfZf
-\frac{\nu^2}{2}\sum_{1\leq l<j\leq 2n}\left(\left(\frac{\delta^l_j+\delta^j_l}{2}\right)Zf\right)^2.
\end{align*}
Since
\[
\Gamma_2(f,f)+\nu\Gamma^Z_2(f,f)\geq
(\rho_1-\frac{1}{\nu})\|u\|^2+(\rho_2-\rho_3\nu^2)\|v\|^2,
\]
by comparing the coefficients of $\nu^2$ terms we have that
\[
\frac{1}{2}\sum_{1\leq l<j\leq 2n}\left(\left(\frac{\delta^l_j+\delta^j_l}{2}\right)Zf\right)^2\leq\frac{c_3}{2w}+\frac{\iota}{4}
\]
for all $w>0$. Let $w\to\infty$ we obtain
\[
\|\tau\|^2=\sum_{l,j=1}^{2n}\left(\left(\frac{\delta^l_j+\delta^j_l}{2}\right)Zf\right)^2\leq\iota.
\]
\end{itemize}
\end{proof}
\section{Stochastic completeness and Bonnet Myers type  theorem}\label{Bonnets}
Throughout this section we assume that $L$ satisfies the generalized curvature dimension condition CD($\rho_1$, $\rho_2$, $\rho_3$, $\kappa$, $\infty$) with $\rho_1\in\R$, $\rho_2>0$, $\rho_3>0$, $\kappa>0$. Our purpose here is to study the stochastic completeness property of the heat semigroup and the compactness properties of the manifold $\M$.

Let us  introduce the rescaled Riemannian metric $$g_\lambda=d\theta(\cdot,J\cdot)+\lambda^{-2}\theta^2,$$where $\lambda>0$. The associated Laplacian $\Delta^\lambda$ is given by
\[
\Delta^\lambda=L+\lambda^2Z^2
\]
and the associated first order bilinear form is
\[
\Gamma^\lambda(f)=\Gamma(f)+\lambda^2(Zf)^2.
\]
\begin{lemma}\label{CDRiem}
If there exists $\alpha, \iota \ge 0$ such that for every $f \in C^\infty(\M)$,
\begin{equation}\label{Ctau}
\langle(\nabla_Z\tau)(\nabla_{\mathcal{H}}f),\nabla_{\mathcal{H}}f\rangle\leq \alpha \| \nabla_\mathcal{H} f\|^2, \quad \|\tau (\nabla_\mathcal{H} f) \|^2\leq\iota \| \nabla_\mathcal{H} f\|^2,
\end{equation} 
then we have
\begin{equation}\label{gammal}
 \Gamma_2^\lambda(f)\geq \Gamma_2(f)+\lambda^2\Gamma_2^Z(f)-\lambda^2\left(2\iota+\alpha \right)\Gamma(f),
\end{equation}
and consequently
\[
 \Gamma_2^\lambda(f)\geq c(\lambda) \Gamma^\lambda(f),
\]
where $c(\lambda)=\min\left\{\rho_1-\frac{\kappa}{\lambda^2}+\lambda^2\left(2\iota+\alpha \right), \frac{\rho_2}{\lambda^2}-\rho_3\lambda^2 \right\}$.
\end{lemma}
\begin{proof}
Some easy computations show that
\begin{eqnarray*}
2\Gamma_2^\lambda(f)&=&\Delta^\lambda\Gamma^\lambda(f)-2\Gamma^\lambda(f,\Delta^\lambda(f))\\
&=&2\Gamma_2(f)+\lambda^2\left(Z^2\Gamma(f)-2\Gamma(f,Z^2f)+2\Gamma_2^Z(f)\right)+2\lambda^4(Z^2f)^2,
\end{eqnarray*}
and, in a local orthonormal frame,
\begin{eqnarray*}
Z^2\Gamma(f)-2\Gamma(f,Z^2f)=2\sum_{k=1}^{2n}(ZX_kf-\sum_{i=1}^{2n}\delta_i^kX_if)^2-2\sum_{i,j,k=1}^{2n}\delta_i^k(\delta_j^k+\delta_k^j)X_ifX_jf-2\sum_{i,k=1}^{2n}(Z\delta_i^k)X_ifX_kf.
\end{eqnarray*}
Since
\[
ZX_kf=X_kZf-\sum_{i=1}^{2n}\delta_k^iX_if
\]
and 
\[
\sum_{i=1}^{2n}\bigg((\nabla_Z\tau)(X_i)\bigg)fX_if=\sum_{i,k=1}^{2n}Z(\delta_i^k)X_ifX_kf+\frac{1}{2}\sum_{i,j,k=1}^{2n}(\delta_i^k\delta_j^k-\delta_k^i\delta_k^j)X_ifX_jf,
\]
we can conclude that
\begin{equation}
\frac{1}{2}Z^2\Gamma(f)-\Gamma(f,Z^2f)=\sum_{k=1}^{2n}(X_kZf-2\tau(X_k)f)^2-2\|\tau(\nabla_{\mathcal{H}}f)\|^2-\langle(\nabla_Z\tau)(\nabla_{\mathcal{H}}f),\nabla_{\mathcal{H}}f\rangle,
\end{equation}
and thus
\[
\frac{1}{2}Z^2\Gamma(f)-\Gamma(f,Z^2f)\geq-2\|\tau(\nabla_{\mathcal{H}}f)\|^2-\langle(\nabla_Z\tau)(\nabla_{\mathcal{H}}f),\nabla_{\mathcal{H}}f\rangle.
\]
At the end we obtain \eqref{gammal} by plugging in \eqref{Ctau}.  The inequality (\ref{gammal}) is  obtained by using the generalized curvature condition CD($\rho_1$, $\rho_2$, $\rho_3$, $\kappa$, $\infty$).
\end{proof}

This lemma has a very interesting first corollary.

\begin{theorem}\label{myers}
 Assume that there exist $\alpha, \iota  \ge 0$ such that for every $f \in C^\infty(\M)$,
\begin{equation*}
\langle(\nabla_Z\tau)(\nabla_{\mathcal{H}}f),\nabla_{\mathcal{H}}f\rangle\leq \alpha \| \nabla_\mathcal{H} f\|^2, \quad \|\tau (\nabla_\mathcal{H} f) \|^2\leq\iota \| \nabla_\mathcal{H} f\|^2,
\end{equation*} 
 and moreover that
  $\rho_1>\sqrt
 {\frac{\rho_3}{\rho_2}}\kappa+\sqrt{\frac{\rho_2}{\rho_3}}(2\iota+\alpha)$, then the manifold $\M$ is compact.
\end{theorem}

\begin{proof}
If $\rho_1>\sqrt
 {\frac{\rho_3}{\rho_2}}\kappa+\sqrt{\frac{\rho_2}{\rho_3}}(2\iota+\alpha)$, then we can chose $\lambda>0$ such that  $c(\lambda) >0$. It implies that the Ricci curvature of the Riemannian metric $g^\lambda$ is bounded from below by a positive number and thus $\M$ is compact from the classical Bonnet-Myers theorem. 
\end{proof}

\begin{remark}
In the Sasakian case, $\alpha=\iota=\rho_3=0$ and  we recover the result from \cite{BG1}. However, in \cite{BG1} the compactness result came with an upper bound for the Carnot-Carath\'eodory diameter of $\M$. 
\end{remark}

A second corollary is  the following volume estimate of the metric balls and the stochastic completeness of the heat semigroup generated by $L$ (see the next Section for a definition). Let us first remind that the distance $d$ associated to $L$ 
is defined by:
\[
d(x,y) =\sup \left\{ f(x)-f(y), f \in C^\infty(\M), \| \Gamma(f) \|_\infty \le 1 \right\}.
\]
It also coincides with the usual Carnot-Carath\'eodory distance.
\begin{theorem}\label{stochastic complete}
Assume that there exist $\alpha, \iota  \ge 0$ such that for every $f \in C^\infty(\M)$,
\begin{equation*}
\langle(\nabla_Z\tau)(\nabla_{\mathcal{H}}f),\nabla_{\mathcal{H}}f\rangle\leq \alpha \| \nabla_\mathcal{H} f\|^2, \quad \|\tau (\nabla_\mathcal{H} f) \|^2\leq\iota \| \nabla_\mathcal{H} f\|^2.
\end{equation*} 
There exist constants $C_1\geq 0$ and $C_2\geq 0$ such that for every $x \in \M$ and every $r \ge 0$
\begin{equation}\label{volume}
\mu(B(x,r))\leq C_1e^{C_2r}.
\end{equation}
As a consequence, the heat semigroup $P_t$ generated by the sub-Laplacian is stochastically complete, that is for every $t \ge 0$, $P_t 1=1$.
\end{theorem}
\begin{proof}
Let $B_\lambda(x,r)$ denote the  $g^\lambda$ Riemannian ball in $\M$ centered at $x$ with  radius $r$. It is easy to see that
\[
B(x,r)\subset B_\lambda (x,r).
\]
By Lemma \ref{CDRiem}, the Ricci curvature of the Riemannian metric $g^\lambda$ is bounded from below. From the Riemannian volume comparison  theorem, we deduce then that $\mu(B(x,r))\leq C_1e^{C_2r}$.
As a consequence, we conclude that for every $x \in \M$,
\[
\int_0^{+\infty}\frac{rdr}{\log{\mu (B(x,r))}}=\infty.
\]
Thanks to a theorem by K.T. Sturm \cite{KTS}, we deduce that $P_t$ is stochastically complete.
\end{proof}

\section{Gradient bounds for the heat semigroup and spectral gap estimates}\label{GB}

In the whole  section, we assume again that the sub-Laplacian of $L$ satisfies the generalized $CD(\rho_1,\rho_2,\rho_3, \kappa, \infty)$ for some constants $\rho_1\in\R$, $\rho_2>0$, $\rho_3>0$, $\kappa>0$. 

The previous section has shown how to deduce some interesting geometric consequences of the generalized curvature dimension condition. However an additional bound is required on the tensor $\nabla_Z \tau$  and the techniques are not intrinsically associated to $L$ in the sense that we introduced the rescaled Riemannian metric $g^\lambda$ and used results from Riemannian geometry. In this Section, we develop tools to exploit in an intrinsic way the generalized curvature dimension inequality. These methods rely on the study of gradient bounds for the subelliptic heat semigroup which is generated by $L$.

We first remind the construction of the heat semigroup associated to $L$. From Lemma \ref{ess}, the operator $L$ is essentially self-adjoint. Let us denote by $L=-\int_0^{+\infty} \lambda dE_\lambda$  the spectral
decomposition of $L$ in $L^2 (\bM,\mu)$. By definition, the
heat semigroup $(P_t)_{t \ge 0}$ is given by $P_t= \int_0^{+\infty}
e^{-\lambda t} dE_\lambda$. It is a one-parameter family of bounded operators on
$L^2 (\bM,\mu)$ which  transforms
positive functions into positive functions and satisfies
\begin{equation}\label{submarkov}
P_t 1 \le 1.
\end{equation}
This property implies in particular \begin{equation}\label{sminfty}
||P_tf||_{L^1(\bM)} \le ||f||_{L^1(\bM)},\ \ \
||P_tf||_{L^\infty(\bM)} \le ||f||_{L^\infty(\bM)},
\end{equation}
and therefore by the  Riesz-Thorin interpolation theorem
\begin{equation}\label{smp}
||P_tf||_{L^p(\bM)} \le ||f||_{L^p(\bM)},\ \  1\le p\le \infty.
\end{equation}

Moreover, it can be shown as in \cite{Li} that $P_t$ is the unique solution in $L^p$ of the parabolic Cauchy problem:

\begin{proposition}\label{uniquenessLp}
The unique solution of the Cauchy problem
\[
\begin{cases}
\frac{\p u}{\p t} - Lu = 0,
\\
u(x,0) = f(x),\ \ \ \   f\in L^p(\bM),1<p<+\infty,
\end{cases}
\]
that satisfies $\| u(\cdot,t) \|_p < +\infty$ is given by $u(x,t)=P_t f(x)$.
\end{proposition}
Due to the hypoellipticity of $L$, the function $(t,x) \rightarrow P_t f(x)$ is
smooth on $\mathbb{M}\times (0,\infty) $ and
\[ P_t f(x)  = \int_{\mathbb M} p(x,y,t) f(y) d\mu(y),\ \ \ f\in
C^\infty_0(\mathbb M),\] where $p(x,y,t) > 0$ is the so-called heat
kernel associated to $P_t$. Such function is smooth and it is symmetric, i.e., \[ p(x,y,t)
= p(y,x,t). \]
 By the
semi-group property for every $x,y\in \bM$ and $s,t>0$ we have
\begin{align}\label{sgp}
p(x,y,t+s) & = \int_\bM p(x,z,t) p(z,y,s) d\mu(z)  
\\
& = \int_\bM p(x,z,t)
p(y,z,s) d\mu(z) = P_s(p(x,\cdot,t))(y).
\notag
\end{align}

In order to use heat semigroup gradient bounds techniques,   we will need  the following hypothesis throughout this section.
\begin{hypothesis}\label{H}
The semigroup $P_t$ is stochastically complete, i.e., for $t>0$, 
\[
P_t1=1,
\]
and for all $f\in C^\infty_0(\M)$ and $T\geq 0$, one has
\[
\sup_{t\in[0,T]}\|\Gamma(P_tf)\|_\infty+\|\Gamma^Z(P_tf)\|_\infty<+\infty.
\]
\end{hypothesis}

The Hypothesis \ref{H} is not very strong.  It is obviously satisfied if $\M$ is compact. In the non compact case, a  general criterion is given in the Appendix. From now on, in this section, we assume that  that Hypothesis \ref{H} is satisfied.

 The \textit{raison d'\^etre} of Hypothesis \ref{H} is the following theorem that was proved in \cite{BG1}.
 
\begin{theorem}\label{p.4.4.3}
Assume that Hypothesis \ref{H} is satisfied. Let $T>0$.  Suppose that $u,v:\M \times[0,T]\rightarrow \mathbb{R}$ are smooth functions such that $\sup_{t\in [0,T]}\| u(\cdot,t)\|_\infty<\infty$ and $\sup_{t\in [0,T]}\| v(\cdot,t)\|_\infty<\infty$.  Suppose
\[Lu+\frac{\partial u}{\partial t}\geq v\]
on $\M\times [0,T]$.  Then for all $x\in \M$,
\[P_T(u(\cdot,T))(x)\geq u(x,0)+\int_0^TP_s(v(\cdot,s))(x)ds.\]
\end{theorem}

We can now prove the main gradient bound for the heat semigroup.

\begin{proposition}\label{prop-sob0lev}
Let us assume $\rho_1-\frac{\kappa\sqrt{\rho_3}}{\sqrt{\rho_2}}\geq 0$. For $f\in C_0^\infty(\M)$ and  $t\ge 0$,   we have 
\[
\Gamma(P_tf)+\frac{\sigma+\sqrt{\sigma^2+16\rho_2\rho_3}}{4\rho_2}\Gamma^Z(P_tf)\leq e^{-\sigma t}\left(P_t(\Gamma(f))+\frac{\sigma+\sqrt{\sigma^2+16\rho_2\rho_3}}{4\rho_2}P_t(\Gamma^Z(f)) \right)
\]
where $\sigma=\frac{2\rho_1\rho_2-2\kappa\sqrt{\rho_2\rho_3}}{(\rho_2+\kappa)}$.
\end{proposition}
\begin{proof}
Let us fix $t>0$ once time for all in the following proof. For $0<s<t$, let 
\begin{align*}
\phi_1(x,s)=\Gamma(P_{t-s}f)(x),\\
\phi_2(x,s)=\Gamma^Z(P_{t-s}f)(x),
\end{align*}
be defined on $\M\times[0,t]$. A simple computation shows that
\begin{align*}
L\phi_1+\frac{\partial\phi_1}{\partial s}=2\Gamma_2(P_{t-s}f),\\
L\phi_2+\frac{\partial\phi_2}{\partial s}=2\Gamma^Z_2(P_{t-s}f),
\end{align*}
Now consider the function
\begin{eqnarray*}
\phi(x,s)=a(s)\phi_1(x,s)+b(s)\phi_2(x,s)
\end{eqnarray*}
Applying the generalized curvature-dimension inequality $CD(\rho_1,\rho_2,\rho_3,\kappa,\infty)$, one obtains
\begin{eqnarray}\label{ineq}
& & L\phi+\frac{\partial\phi}{\partial s}=a'\Gamma(P_{t-s}f)+b'\Gamma^Z(P_{t-s}f)+2a\Gamma_2(P_{t-s}f)+2b\Gamma_2^Z(P_{t-s}f)  \nonumber\\
&\geq&\left(a'+2\rho_1a-2\kappa\frac{a^2}{b} \right)\Gamma(P_{t-s}f)+\left( b'+2\rho_2a-2\rho_3\frac{b^2}{a}\right)\Gamma^Z(P_{t-s}f).
\end{eqnarray}
Let us chose
\[
b(s)=e^{\frac{-2\rho_1\rho_2+2\kappa\sqrt{\rho_2\rho_3}}{(\rho_2+\kappa)}s},
\]
and 
\[
a(s)=\frac{\sigma+\sqrt{\sigma^2+16\rho_2\rho_3}}{4\rho_2}b(s),
\]
where $\sigma=\frac{2\rho_1\rho_2-2\kappa\sqrt{\rho_2\rho_3}}{(\rho_2+\kappa)}$, and denote $\delta=\frac{\sigma+\sqrt{\sigma^2+16\rho_2\rho_3}}{4\rho_2}$. It is easy to observe that
\[
b'(s)=-\sigma b(s), \quad a'(s)=-\sigma a(s)=-\sigma\delta b(s).
\]
We now claim that $a(s)$, $b(s)$ satisfy 
\begin{eqnarray}
a'+2a\rho_1-2\kappa\frac{ a^2}{b}\geq0, \label{eqa} \\
b'+2a\rho_2-2\rho_3\frac{b^2}{a}=0.  \label{eqb}
\end{eqnarray}
Indeed, (\ref{eqb}) writes as
\[
-\sigma\delta+2\delta^2\rho_2-2\rho_3=0,
\]
and follows immediately by the expressions of $\delta$. To see (\ref{eqa}), similarly, we only need to prove that
\[
-\delta \sigma+2\rho_1\delta -2\kappa\delta^2\geq 0,
\]
which is equivalent to prove
\[
2\rho_1\geq2\kappa\delta+\sigma.
\]
We can obtain it by observing that
\[
\kappa\sqrt{\sigma^2+16\rho_2\rho_3}\leq 4\kappa\sqrt{\rho_2\rho_3}+\kappa\sigma,
\]
thus we have claim proved.  Plug (\ref{eqa}) and (\ref{eqb}) into (\ref{ineq}), we get
\[
L\phi+\frac{\partial \phi}{\partial s}\geq 0
\] 
and by the  comparison result of Theorem \ref{p.4.4.3}, we have that 
\[
P_t(\phi(\cdot,t))(x)\geq\phi(x,0).
\]
We complete the proof by realizing that
\[
\phi(x,0)=a(0)\Gamma(P_tf)(x)+b(0)\Gamma^Z(P_tf)(x),
\]
and 
\[
P_t(\phi(\cdot,t))(x)=a(t)P_t(\Gamma (f))(x)+b(t)P_t(\Gamma^Z (f))(x).
\]
\end{proof}

A direct application of the above inequality is the fact $\rho_1-\frac{\kappa\sqrt{\rho_3}}{\sqrt{\rho_2}} > 0$ implies that the invariant measure is finite.

\begin{corollary}
If $\rho_1-\frac{\kappa\sqrt{\rho_3}}{\sqrt{\rho_2}} > 0$ then $\M$ has a finite volume, i.e., 
\[
\mu(\M)<+\infty.
\]
\end{corollary}

\begin{proof}
Let $f,g\in C^\infty_0(\M)$, and write
\[
\int_\M(P_tf-f)gd\mu=\int_\M\int_0^t\frac{\partial{(P_sf)}}{\partial s}gdsd\mu=\int_0^t\int_\M (LP_sf)gd\mu ds=-\int_0^t\int_\M\Gamma(P_sf,g)d\mu ds
\]
By Cauchy-Schwartz inequality, we have
\[
\bigg\vert\int_\M(P_tf-f)gd\mu\bigg\vert\leq \int_0^t\int_\M\left(\Gamma(P_sf)^{\frac{1}{2}}\Gamma (g)^{\frac{1}{2}} \right)d\mu ds.
\]
Applying Proposition \ref{prop-sob0lev}, we obtain then
\[
\bigg\vert\int_\M(P_tf-f)gd\mu\bigg\vert\leq
\left(\int_0^t e^{\frac{\rho_1\rho_2-2\kappa\sqrt{\rho_2\rho_3}}{(\rho_2+\kappa)}s}ds
\int_\M\Gamma(g)^{\frac{1}{2}} d\mu\right)
\sqrt{\|\Gamma(f)\|_\infty+\frac{\sigma+\sqrt{\sigma^2+16\rho_2\rho_3}}{4\rho_2}\|\Gamma^Z(f)\|_\infty},
\]
where $\sigma=\frac{2\rho_1\rho_2-2\kappa\sqrt{\rho_2\rho_3}}{(\rho_2+\kappa)}$.

Moreover, from the spectral theorem we know that $P_tf$ converges to $P_\infty f$ in $L^2(\M)$ and $LP_\infty f=0$, where $P_\infty f$ is in the domain of $L$. Hence $\Gamma(P_\infty f)=0$, which implies that $P_\infty f$ is a constant.

We then prove the measure is finite by contradiction. 
Assume $\mu(\M)=+\infty$, then we have  $P_\infty f=0$, thus when $t\to+\infty$,
\[
\bigg\vert\int_\M fgd\mu\bigg\vert\leq
\left(\int_0^{+\infty} e^{\frac{\rho_1\rho_2-2\kappa\sqrt{\rho_2\rho_3}}{(\rho_2+\kappa)}s}ds
\int_\M\Gamma(g)^{\frac{1}{2}} d\mu\right)
\sqrt{\|\Gamma(f)\|_\infty+\frac{\sigma+\sqrt{\sigma^2+16\rho_2\rho_3}}{4\rho_2}\|\Gamma^Z(f)\|_\infty}. 
\]
Let $g\geq0$, $g\not=0$, and we chose for $f$  an increasing sequence $\lbrace h_k\rbrace \in C^\infty_0(\M)$ such that
$h_k \nearrow1$ on $\M$ and  
\[
\|\Gamma(h_k)\|_\infty+\|\Gamma^Z(h_k)\|_\infty\to_{k\to+\infty} 0.
\]
By letting $k \to +\infty$, we obtain
\[
\int_\M gd\mu\leq 0,
\]
which is a contradiction. Hence $\mu(\M)<+\infty$. \end{proof}

Another corollary is the following Poincar\'e inequality.

\begin{corollary}\label{Poincare}
If $\rho_1-\frac{\kappa\sqrt{\rho_3}}{\sqrt{\rho_2}} > 0$, then for all $f\in C^\infty_0(\M)$, 
\begin{equation}\label{spectral}
\int_\M f^2d\mu-\left(\int_\M fd\mu \right)^2\leq\frac{\rho_2+\kappa}{\rho_1\rho_2-\kappa\sqrt{\rho_2\rho_3}}\int_\M\Gamma(f)d\mu.
\end{equation}
\end{corollary}
\begin{proof}
By proposition \ref{prop-sob0lev}, we have
\begin{eqnarray}\label{grad1}
\int_\M\Gamma(P_tf)d\mu &\leq&  e^{\frac{2\rho_1\rho_2-2\kappa\sqrt{\rho_2\rho_3}}{(\rho_2+\kappa)}t}\int_\M\left(P_t(\Gamma(f))+\frac{\sigma+\sqrt{\sigma^2+16\rho_2\rho_3}}{4\rho_2}P_t(\Gamma^Z(f)) \right)d\mu \nonumber\\
&\leq& e^{\frac{2\rho_1\rho_2-2\kappa\sqrt{\rho_2\rho_3}}{(\rho_2+\kappa)}t}\int_\M\left(\Gamma(f)+\frac{\sigma+\sqrt{\sigma^2+16\rho_2\rho_3}}{4\rho_2}\Gamma^Z(f) \right)d\mu,
\end{eqnarray}
where the last inequality is due to the contractivity of $P_t$. Let $dE_\lambda$ be  the  spectral resolution of $-L$. Then by the spectral theorem we have
\begin{equation}\label{grad2}
\int_\M\Gamma(P_tf)d\mu=\int_0^{+\infty}\lambda e^{-2\lambda t}dE_\lambda(f)
\end{equation}
and 
\[
\int_\M\Gamma(f)d\mu=\int_0^{+\infty}\lambda dE_\lambda(f).
\]
Thus for $0<s<t$, by H\"older inequality, 
\begin{equation}\label{grad3}
\int_\M\Gamma(P_sf)d\mu=\int_0^{+\infty}\lambda e^{-2\lambda s}dE_\lambda(f)\leq \left(\int_0^{\infty}\lambda e^{-2\lambda t}dE_\lambda(f) \right)^{\frac{s}{t}}\left(\int_0^\infty\lambda dE_\lambda(f) \right)^{\frac{t-s}{t}}.
\end{equation}
We  denote $C(f)=\int_\M\left(\Gamma(f)+\frac{\sigma+\sqrt{\sigma^2+16\rho_2\rho_3}}{4\rho_2}\Gamma^Z(f) \right)d\mu$, then by \eqref{grad1}, \eqref{grad2} and \eqref{grad3} we have
\[
\int_\M\Gamma(P_sf)d\mu\leq e^{\frac{2\rho_1\rho_2-2\kappa\sqrt{\rho_2\rho_3}}{(\rho_2+\kappa)}s}C(f)^{\frac{s}{t}}\left(\int_\M\Gamma(f)d\mu\right)^{\frac{t-s}{t}}.
\]
By letting $t\to +\infty$, we obtain
\[
\int_\M\Gamma(P_sf)d\mu\leq e^{\frac{2\rho_1\rho_2-2\kappa\sqrt{\rho_2\rho_3}}{(\rho_2+\kappa)}s}\int_\M\Gamma(f)d\mu.
\]
At the end, we obtain the desired Poincar\'e inequality by observing 
\[
\int_\M f^2d\mu-\left(\int_\M fd\mu \right)^2=-\int_0^\infty \frac{\partial}{\partial s}\int_\M (P_sf)^2 d\mu ds=\int_\M\Gamma(P_sf)d\mu.
\]
\end{proof}

This result naturally raises the conjecture that if  $\rho_1-\frac{\kappa\sqrt{\rho_3}}{\sqrt{\rho_2}} > 0$, then $\M$ is compact. This would be a stronger result than Theorem \ref{myers}.

\section{Appendix: Gradient bounds by stochastic analysis}\label{appendix}

The goal of the section is to study some general conditions ensuring that Hypothesis \ref{H} is true. 

Let $\M$ be a $n+m$ dimensional smooth manifold. We assume given $n+m$ smooth vector fields  $\{X_1,\cdots,X_{n+m} \}$ on $\M$ such that for every $x\in \M$, $\{X_1(x),\cdots,X_{n+m}(x) \}$ is a basis of $T_x \M$. This global basis of vector fields induce on $\M$ a Riemannian metric $g$ that we assume to be complete. There exist smooth  functions $\omega_{ij}^k:\M \to \mathbb{R}$, $i,j,k=1, \cdots, n+m$, such that:
\[
[X_i,X_j]=\sum_{k=1}^{n+m} \omega_{ij}^k X_k.
\]
 We  assume that the vector fields $\{X_1,\cdots,X_n\}$ satisfy the H\"ormander's bracket generating condition.

Let us consider the symmetric and subelliptic operator
\[
L=-\frac{1}{2} \sum_{i=1}^n X_i^*X_i,
\]
where $X_i^*=-X_i + \mathbf{div} X_i$ is the formal adjoint of $X_i$ with respect to the Riemannian measure $\mu$.  By using a similar argument as in the proof of Lemma \ref{ess}, it is seen that the assumed completeness of $g$ implies that $L$ is essentially self-adjoint on the space $C_0^\infty(\M)$. As a consequence, $L$ is the generator of sub-Markov semigroup $(P_t)_{t \ge 0}$. Let us observe that $L$ can also be written as
\[
L=X_0+\frac{1}{2} \sum_{k=1}^n X_k^2,
\]
where $X_0=-\frac{1}{2} \sum_{i=1}^n (\mathbf{div} X_i) X_i=-\frac{1}{2}\sum_{i=1}^n\sum_{k=1}^{n+m} \omega_{ik}^k X_i$. We thus can find some smooth functions $\omega_{0i}^k$'s such that
\[
[X_0,X_i]=\sum_{k=1}^{n+m} \omega_{0i}^k X_k.
\]
 Let now $(B_t)_{t \ge 0}$ be a $n$-dimensional Brownian motion.

If we consider the stochastic differential equation on $\M$,
\[
dY_t^x=\sum_{k=0}^{n} X_k(Y_t^x)\circ dB^k_t, \quad Y_0^x=x,
\]
with the notation $B^0_t=t$, it has a unique solution defined up to an explosion time $\mathbf{e} (x)$. If $f$ is a bounded Borel function on $\M$, we then have the representation
\[
P_t f(x) =\mathbb{E}\left( f(Y_t^x) 1_{t <\mathbf{e} (x)} \right).
\]

Our goal is to prove the following theorem:

\begin{theorem}
Let us assume that the functions $ \omega_{ij}^k $, $X_l  \omega_{ij}^k $, $i,j,k,l=1,\cdots, n+m$ are bounded, then the semigroup $P_t$ is stochastically complete and there exist constants $C_1,C_2 \ge 0$ such that for every $f \in C_0^\infty(\M)$, $t \ge 0$ and $x \in \M$
\[
\sum_{k=1}^{n+m} (X_k P_t f)^2(x) \le C_1 e^{C_2 t} \left( \sum_{k=1}^{n+m} \| X_k f \|_\infty^2 \right).
\]
\end{theorem}

\begin{proof}
We adapt some ideas from Kusuoka \cite{K}.  Let $x,y \in \M$ and let $\mathcal{O}$ be a bounded open set that contains the Riemannian geodesic connecting $x$ to $y$. Let $R>0$ such that the ball $B(x,R)$ with center $x$ and  radius $R$ contains $\mathcal{O}$. We denote
\[
T_R =\inf_{z \in\overline{ \mathcal{O}}} \inf  \{ t \ge 0, Y_t^z \notin B(x,R) \}.
\]
Let us then consider for $f \in C_0^\infty(\M)$, and $z \in \mathcal{O}$,
\[
P_t^Rf(z)=\mathbb{E}\left(f(Y^z_{t \wedge T_R}) \right).
\]
By using the chain rule, and the triangle inequality, we see that for $z \in  \mathcal{O}$,
\[
\sum_{k=1}^{n+m} (X_k P^R_t f)^2(z ) \le \mathbb{E} \left( \| J_{t\wedge T_R}^*(z) \nabla f (Y_{t\wedge T_R}^z)\|\right)^2\le \mathbb{E} \left( \| J_{t\wedge T_R}^*(z) \|\right)^2\left( \sum_{k=1}^{n+m} \| X_k f \|_\infty^2 \right).
\]
where $J_t(z)=\frac{\partial Y_t^z}{\partial z}, t <T_R$, is the first variation process of the stochastic differential equation and $J^*$ the adjoint matrix. We thus want to find a bound for $ \mathbb{E} \left( \| J_{t\wedge T_R}^*(z) \|\right)$ that does not depend on $R$ and $z$. Since $\{X_1,\cdots,X_{n+m} \}$ form a basis at each point, we can find processes $\beta^k_i(t,z)$, $k=1,\cdots, m+n$, $i=1,\cdots n$ such that for $t <T_R$,
\[
J_t^{-1}(X_i(Y_t^z))=\sum_{k=1}^{m+n} \beta^k_i(t,z) X_k(z).
\]
By using the chain rule, we have for $t<T_R$,
\begin{align*}
dJ_t^{-1}(X_i(Y_t^z)) &=\sum_{k=0}^{n} J_t^{-1}([X_k,X_i](Y_t^z)) \circ dB^k_t \\
 &=\sum_{k=0}^{n} \sum_{l=1}^{m+n} \omega_{ki}^l(Y_t^z) J_t^{-1}(X_k(Y_t^z))\circ dB^k_t.
\end{align*}
As a consequence the matrix valued process $\beta(t,z)$, $t <T_R$ solves the matrix stochastic differential equation,
\[
d\beta(t,z) =\sum_{k=0}^{n} \omega_k(Y_t^z) \beta(t,z) \circ dB^k_t.
\]
The inverse matrix process $\alpha(t,z)=\beta(t,z)^{-1}$ will then solve the linear stochastic differential equation for $t <T_R$,
\[
d\alpha(t,z) =-\sum_{k=0}^{n}\alpha(t,z)  \omega_k(Y_t^z) \circ dB^k_t.
\]
From our assumption, all the coefficients of the equation are bounded. We therefore obtain a bound $\mathbb{E}(\| \alpha(t,z)\|)\le C_1 e^{C_2 t}$, where $C_1,C_2$ are independent from $R$ and $z$. As a conclusion, we get
\[
\sum_{k=1}^{n+m} (X_i P^R_t f)^2(z ) \le  C_1 e^{C_2 t} \left( \sum_{k=1}^{n+m} \| X_k f \|_\infty^2 \right).
\]
By integrating the inequality over the geodesic between $x$ and $y$, we obtain
\[
| (P^R_t f)(x)- (P^R_t f)(y) |^2  \le  C_1 e^{C_2 t} \left( \sum_{k=1}^{n+m} \| X_k f \|_\infty^2 \right) d(x,y)^2.
\]
We can then let $R \to + \infty$ to conclude
\[
| (P_t f)(x)- (P_t f)(y) |^2  \le  C_1 e^{C_2 t} \left( \sum_{k=1}^{n+m} \| X_k f \|_\infty^2 \right)d(x,y)^2.
\]
Since this is true for every $x,y \in \M$, we conclude
\[
\sum_{k=1}^{n+m} (X_k P_t f)^2(x ) \le  C_1 e^{C_2 t} \left( \sum_{k=1}^{n+m} \| X_k f \|_\infty^2 \right).
\]
We now prove the stochastic completeness. Let $f,g \in  C^\infty_0(\mathbb M)$, we have
\begin{align*}
\int_{\bM} (P_t f -f) g d\mu = \int_0^t \int_{\bM}\left(
\frac{\partial}{\partial s} P_s f \right) g d\mu ds= \int_0^t
\int_{\bM}\left(L P_s f \right) g d\mu ds=- \int_0^t \int_{\bM}
\Gamma ( P_s f , g) d\mu ds.
\end{align*}
By means of Cauchy-Schwarz inequality  we
find
\begin{equation}\label{P1}
\left| \int_{\bM} (P_t f -f) g d\mu \right| \le \left(\int_0^t C_1e^{C_2s}
ds\right)  \| \nabla f \|_\infty  \int_{\bM}\Gamma (g)^{\frac{1}{2}}d\mu.
\end{equation}
We now apply \eqref{P1} with $f = h_k$, where $h_k$ is a sequence such that $h_k\nearrow 1$, $h_k \ge 0$ and $ \sum_{k=1}^{n+m} \| X_k h_l\|_\infty^2\to 0$ when $l \to + \infty$.

By Beppo Levi's monotone convergence theorem we have $P_t
h_k(x)\nearrow P_t 1(x)$ for every $x\in \bM$. We conclude that the
left-hand side of \eqref{P1} converges to $\int_{\bM} (P_t 1 -1) g d\mu$. Since  the right-hand side converges to zero, we reach the conclusion
\[
\int_{\bM} (P_t 1 -1) g d\mu=0,\ \ \ g\in C^\infty_0(\bM).
\]
It follows that $P_t 1 =1$.
\end{proof}

\end{document}